\documentclass[11pt,reqno]{amsart}

\usepackage{amsmath,amssymb,amsthm,amsfonts}
\usepackage{cite}
\usepackage{xcolor}
\allowdisplaybreaks

\theoremstyle{plain}
\newtheorem{thm}{Theorem}[section]
\newtheorem{defn}[thm]{Definition}
\newtheorem{lemma}[thm]{Lemma}
\newtheorem{proposition}[thm]{Proposition}
\newtheorem{remark}[thm]{Remark}
\newtheorem{corollary}[thm]{Corollary}
\newtheorem{theorem}[thm]{Theorem}

\usepackage{mathrsfs}

\numberwithin{equation}{section}

\newcommand{\R}{\mathbb{R}}
\newcommand{\nn}{{n}}
\newcommand{\tanproj}{\text{tan}}
\newcommand{\Ltwosigma}{L^2_{\sigma,\tan}(\Omega)}
\newcommand{\Honesigma}{H^{1}_{\sigma,\tan}(\Omega)}
\newcommand{\boussop}{\mathbb{B}} 

\newcommand{\navierbc}{\left[(2S({u}))\nn + \alpha {u}\right]_{\tanproj} = 0}
\DeclareMathOperator{\kernel}{Ker}
\DeclareMathOperator{\Proj}{Proj}
\title[Decay of the 3D Boussinesq system]{Exponential decay for the 3D Boussinesq equations with Navier boundary conditions}
\author[Wen Feng, Weinan Wang and Xiang Xu]{Wen Feng$^{1}$, Weinan Wang$^{2}$ and  Xiang Xu$^{3}$}
	
	\address{$^1$ Department of Mathematics and Statistics, Sam Houston State University, Huntsville,  Texas, 77340,  United States}
	
	\email{wxf008@shsu.edu}
	
	\address{$^2$ Department of Mathematics, University of Oklahoma, Norman, OK 
		73019, United States}
	
	\email{ww@ou.edu}
	
	\address{$^3$ Department of Mathematics, University of Oklahoma, Norman, OK 
		73019, United States}
	
	\email{xiang.xu-1@ou.edu}

\date{\today}

\begin{document}
\begin{abstract}
We study the three-dimensional incompressible Boussinesq equations on a bounded
domain with smooth boundary and Navier boundary conditions. We construct global weak
solutions by a Galerkin approximation and establish the associated Leray--Hopf
energy inequalities. For nonnegative boundary friction, the total energy
decays exponentially when the friction coefficient is positive on a
boundary subset of positive surface measure. In the frictionless case, the
scalar field and the velocity component orthogonal to the rigid-motion kernel
decay exponentially; when the kernel is trivial, this is exponential decay of
the total energy. 
When the scalar initial datum vanishes, this also proves exponential decay
for the Navier--Stokes system on solids of revolution for every friction
coefficient $\alpha\in L^\infty(\partial\Omega)$ such that $\alpha\ge0$ almost
everywhere and $\alpha\not\equiv0$, resolving the
corresponding case left open in \cite{Kelliher2025}. The proof uses a weighted Korn--Poincar\'e
inequality and a two-time Gronwall-type inequality with an exponentially
decaying forcing term.
\end{abstract}
\maketitle

\section{Introduction}

The long-time behavior of solutions is a fundamental question in the analysis of viscous incompressible fluids. For the Navier--Stokes equations (NSE) on a bounded domain with the classical no-slip condition $(u=0$ on $\partial\Omega)$, exponential decay to the rest state follows from the energy inequality and the Poincar\'e inequality. The zero trace excludes undamped rigid motions and controls the kinetic energy by the viscous dissipation, yielding a uniform decay rate.

The Boussinesq system couples the NSE with a transport--diffusion equation for density or temperature. It is widely used in the analysis of atmospheric and oceanic circulation, as well as heat transfer problems. The literature includes global regularity and well-posedness results for partially dissipative and fractional models \cite{KW2,W1,stefanov1,stefanov2,WJ3,WJ4}; stability, instability, decay, and long-time behavior near equilibria \cite{D1,D2,KW1,WJ2,WJ6,WJ7,WJ8}; and bounded-domain or boundary-condition problems \cite{di,ragusa,W3,WJ5,HuEtAl2018,BleitnerCarlsonNobili2025}. {Related Rayleigh--Taylor instability for an inhomogeneous incompressible geophysical-fluid model is studied in \cite{han1}.} Related mild-solution theories on non-Euclidean spaces appear in \cite{pham1,xuan}, while norm-inflation phenomena are studied in \cite{W2}. Navier proposed the classical slip law in a memoir presented to the Acad\'emie Royale des Sciences in 1822 and later published in volume~6 of its \emph{M\'emoires} \cite{navier}. Such conditions allow the fluid to slip at the boundary, with tangential velocity proportional to tangential stress. They are used to model, for example, flow over rough surfaces, microfluidic flows, and motion near contact lines. Mathematically, they create a subtle dependence of the decay properties on the geometry of the domain and the friction coefficient $\alpha$.

In a recent paper \cite{Kelliher2025}, the authors gave a comprehensive analysis of the three-dimensional Navier--Stokes equations with Navier boundary conditions in smooth bounded domains, with continuous nonnegative friction. They proved exponential decay when the friction is positive at every boundary point, and also when the friction is nonnegative and the kernel of the symmetric-gradient operator is trivial. In the frictionless case with nontrivial kernel, they proved convergence to a rigid rotation. Here and throughout the paper, we denote the symmetric gradient by
\[
S(u):=\frac{\nabla u+(\nabla u)^T}{2},
\]
and define
\[
\kernel S
:=
\left\{
w\in H^1_{\sigma,\tan}(\Omega):S(w)=0
\right\}.
\]
Thus, $\kernel S$ consists of the rigid motions that are tangent to
$\partial\Omega$.

Here $H^1_{\sigma,\tan}(\Omega)$ denotes the space of vector fields in
$H^1(\Omega;\mathbb R^3)$ that are divergence-free in $\Omega$ and have
vanishing normal trace on $\partial\Omega$.

In this paper, we extend this theory to the Boussinesq system. At the $L^2$-energy level, incompressibility and the impermeability boundary condition eliminate the transport contribution. Consequently, the scalar $L^2$-norm decays independently of the velocity estimate, while the buoyancy term $\rho e_3$ acts as an exponentially decaying forcing term in the velocity equation.

Our weighted Korn--Poincar\'e inequality controls the full velocity whenever the nonnegative friction coefficient is positive on a boundary subset of positive surface measure, without requiring triviality of $\kernel S$ or an additional geometric restriction on the domain. This strictly weakens the positive-friction hypothesis in \cite{Kelliher2025}; it also lowers the regularity required of the friction coefficient from continuity to essential boundedness. With zero scalar initial data, it yields a Navier--Stokes decay result that resolves the nontrivial-friction portion of the solid-of-revolution case identified as open in Section~8 of \cite{Kelliher2025}. Let
\[
A_{\partial\Omega}:=\{x\in\partial\Omega:\alpha(x)>0\}.
\]
Here and below, $\mathcal H^2(A_{\partial\Omega})>0$ means positive surface measure. This condition is independent of the representative of $\alpha\in L^\infty(\partial\Omega)$. Since $\alpha\ge0$, either $\alpha=0$ almost everywhere or $\mathcal H^2(A_{\partial\Omega})>0$. Thus Sections~\ref{sec5} and~\ref{sec6} cover every nonnegative friction coefficient in $L^\infty(\partial\Omega)$, and no smallness assumption is imposed on the initial data. Related work established local existence of strong solutions and global stability of large strong solutions under a suitable integral condition \cite{W3}. In Appendix~\ref{appx1}, we prove a two-time Gronwall-type inequality inspired by \cite[Appendix A]{Kelliher2025}. No sign condition on the unknown function is needed, because the proof passes through a shifted quantity that need not have a fixed sign.

The remainder of the work is organized as follows. In Section~\ref{sec2}, we introduce the necessary background. In Section~\ref{sec3}, we construct global weak solutions by a Galerkin approximation adapted to the Navier boundary condition. In Section~\ref{sec4}, we derive the corresponding Leray--Hopf energy inequalities and the weighted Korn--Poincar\'e inequality. Appendix~\ref{appx1} contains the two-time Gronwall inequality used in the decay argument. In Section~\ref{sec5}, we prove exponential decay of the total energy when $\mathcal H^2(A_{\partial\Omega})>0$. In Section~\ref{sec6}, we treat the frictionless case and prove exponential decay of the scalar variable together with the velocity component orthogonal to $\kernel S$; when the kernel is trivial, this gives decay of the full energy.

\subsection{The Boussinesq System with Navier Boundary Conditions}

Let $\Omega \subset \R^3$ be a bounded, connected domain with smooth boundary $\partial\Omega$. Let $\nu, \kappa > 0$ be the viscosity and thermal diffusivity, and let $e_3=(0,0,1)$. For an arbitrary finite $T>0$, we consider the system:
\begin{align}
    \partial_t {u} - \nu \Delta {u} + ({u} \cdot \nabla) {u} + \nabla p &= \rho {e}_3 \quad \text{in } (0, T) \times \Omega, \label{eq:bouss_momentum} \\
    \partial_t \rho - \kappa \Delta \rho + {u} \cdot \nabla \rho &= 0 \quad \text{in } (0, T) \times \Omega, \label{eq:bouss_density} \\
    \nabla \cdot {u} &= 0 \quad \text{in } (0, T) \times \Omega, \label{eq:bouss_div}
\end{align}
with the initial conditions
\begin{align*}
    {u}(0, \cdot) = {u}_0 \quad \text{and} \quad \rho(0, \cdot) = \rho_0,
\end{align*}
and the Navier boundary conditions
\begin{align}
    {u} \cdot \mathbf n = 0, \quad \navierbc, \quad \text{and} \quad \rho = 0 \quad \text{on } (0, T) \times \partial\Omega. \label{eq:navier_bc}
\end{align}

Above, $\alpha\in L^\infty(\partial\Omega)$ is a nonnegative friction coefficient, $\mathbf n$ is the outward unit normal to $\partial\Omega$, and the subscript ``tan'' denotes orthogonal projection onto the tangent space. The homogeneous Dirichlet condition on $\rho$ is essential for decay to zero; under a homogeneous Neumann condition, one must additionally impose zero mean or subtract the conserved mean.

\section{Preliminaries}\label{sec2}

We begin by defining the spaces used in the weak formulation:
\[
\begin{aligned}
C^\infty_{\sigma,\tan}(\overline\Omega)
&:=\{u\in C^\infty(\overline\Omega;\mathbb R^3):
\nabla\cdot u=0,\ u\cdot n=0\ \text{on }\partial\Omega\},\\
{L^2_{\sigma,\tan}(\Omega)}
&{:=\overline{C^\infty_{\sigma,\tan}(\overline\Omega)}^{\,L^2(\Omega;\mathbb R^3)},}\\
H^1_{\sigma,\tan}(\Omega)&:=H^1(\Omega;\mathbb R^3)\cap L^2_{\sigma,\tan}(\Omega).
\end{aligned}
\]
{If \(u_n\in C^\infty_{\sigma,\tan}(\overline\Omega)\)
and \(u_n\to u\) in \(L^2(\Omega;\mathbb R^3)\), then
\(\nabla\cdot u=0\) in \(\mathcal D'(\Omega)\). 
Here, \(\mathcal D'(\Omega)\) denotes the space of distributions on \(\Omega\).
We use
\[
H(\operatorname{div};\Omega)
:=\{v\in L^2(\Omega;\mathbb R^3):\nabla\cdot v\in L^2(\Omega)\},
\]
endowed with its graph norm.
Moreover,
\(u_n\to u\) in \(H(\operatorname{div};\Omega)\) because
\(\nabla\cdot u_n=\nabla\cdot u=0\). Continuity of the weak normal-trace
map from \(H(\operatorname{div};\Omega)\) to
\(H^{-1/2}(\partial\Omega)\) therefore gives \(u\cdot n=0\) on
\(\partial\Omega\) in the weak normal-trace sense.}
{Since
\(C^\infty_{\sigma,\tan}(\overline\Omega)
\subset H^1_{\sigma,\tan}(\Omega)
\subset L^2_{\sigma,\tan}(\Omega)\),}
\[
{
\overline{H^1_{\sigma,\tan}(\Omega)}^{\,L^2(\Omega;\mathbb R^3)}
=
L^2_{\sigma,\tan}(\Omega).}
\]
Thus the form domain is dense in the phase space used in
the spectral construction below. Here ``form domain'' refers to
\(H^1_{\sigma,\tan}(\Omega)\), the domain of the velocity bilinear form,
whereas ``phase space'' refers to \(L^2_{\sigma,\tan}(\Omega)\).
Let
\[
\mathcal{H}:=\Ltwosigma\times L^2(\Omega), \quad \mathcal{V}:=\Honesigma\times H^1_0(\Omega).
\]
We denote by $\mathcal V'$ the continuous dual of $\mathcal V$; with
$\mathcal H$ identified with its dual, we use the Gelfand triple
\[
\mathcal V\hookrightarrow\mathcal H\cong\mathcal H'\hookrightarrow\mathcal V'.
\]
They are endowed with the natural inner products
\[
\langle (u,\rho),(v,\phi) \rangle_{\mathcal{H}}=\int_\Omega u\cdot {v}\,dx+\int_\Omega\rho\phi\,dx,
\]
\[
\langle({u},\rho),({v},\phi)\rangle_{\mathcal{V}}=\int_\Omega\nabla u : \nabla {v}\,dx+\int_\Omega\nabla\rho\cdot\nabla \phi\,dx.
\]
For $u\in H^1_{\sigma,\tan}(\Omega)$, each component has zero mean, since
\[
\int_\Omega u_j\,dx
=\int_\Omega u\cdot\nabla x_j\,dx
=\int_{\partial\Omega}x_j u\cdot n\,dS
-\int_\Omega x_j\nabla\cdot u\,dx
=0.
\]
The Poincar\'e inequality therefore gives
\[
\|u\|_{L^2(\Omega)}\le C_\Omega\|\nabla u\|_{L^2(\Omega)},
\]
so the gradient norm is equivalent to the usual $H^1$ norm on $H^1_{\sigma,\tan}(\Omega)$. We next verify the density of $C^\infty_{\sigma,\tan}(\overline\Omega)$ in $H^1_{\sigma,\tan}(\Omega)$. Let $u\in H^1_{\sigma,\tan}(\Omega)$. By the characterization above,
$\nabla\cdot u=0$ in $\Omega$ and $u\cdot n=0$ on $\partial\Omega$ in the
normal-trace sense. Choose $\widetilde u_n\in C^\infty(\overline\Omega;\mathbb R^3)$ with $\widetilde u_n\to u$ in $H^1$. Continuity of the divergence and trace maps gives
\[
\nabla\cdot\widetilde u_n\to0\quad\text{in }L^2(\Omega),
\qquad
\widetilde u_n\cdot n\to0\quad\text{in }H^{1/2}(\partial\Omega).
\]
The divergence theorem gives the compatibility condition for the Neumann problems
\[
\Delta\psi_n=\nabla\cdot\widetilde u_n\quad\text{in }\Omega,
\qquad
\partial_n\psi_n=\widetilde u_n\cdot n\quad\text{on }\partial\Omega,
\qquad
\int_\Omega\psi_n\,dx=0.
\]
Elliptic regularity gives $\psi_n\in C^\infty(\overline\Omega)$, and we set $u_n:=\widetilde u_n-\nabla\psi_n$. The Neumann estimate
\[
\|\psi_n\|_{H^2(\Omega)}
\le C\left(
\|\nabla\cdot\widetilde u_n\|_{L^2(\Omega)}
+\|\widetilde u_n\cdot n\|_{H^{1/2}(\partial\Omega)}
\right)
\longrightarrow0
\]
gives $u_n\to u$ in $H^1$. The defining equations for $\psi_n$ give
\[
\nabla\cdot u_n=\nabla\cdot\widetilde u_n-\Delta\psi_n=0,
\qquad
u_n\cdot n=\widetilde u_n\cdot n-\partial_n\psi_n=0.
\]
Thus $u_n\in C^\infty_{\sigma,\tan}(\overline\Omega)$ and $u_n\to u$ in
$H^1(\Omega;\mathbb R^3)$, which proves the asserted density.

The following geometric boundary identity from \cite{Kelliher2025}
uses $d\boldsymbol n_x:T_x(\partial\Omega)\to T_x(\partial\Omega)$, the
differential of the outward unit normal (the Gauss map) at $x$; with our sign
convention, $-d\boldsymbol n_x$ is the shape operator.

\begin{lemma}[\cite{Kelliher2025}, Lemma 2.1] \label{lem:boundary_identity}
Let $u, v \in C^\infty_{\sigma,\tan}(\overline{\Omega})$. If $u$ satisfies $\navierbc$, then on $\partial\Omega$ we have
\[
v\cdot(\nabla(u)\nn)=v\cdot [d\nn(u)-\alpha u].
\]
\end{lemma}
The weak formulation is defined directly through the bilinear form and
therefore does not require a pointwise interpretation of the Navier condition
for $\alpha\in L^\infty(\partial\Omega)$. For each $x\in\partial\Omega$, the
map $d\boldsymbol n_x$ is self-adjoint on the tangent space, and the
eigenvalues of $-d\boldsymbol n_x$ are the principal curvatures $k_1(x)$ and
$k_2(x)$. We set
\[
\Lambda_\Omega:=\sup_{x\in\partial\Omega}\max\{|k_1(x)|,|k_2(x)|\}.
\]

The associated dissipation form $\mathscr{B}: \mathcal{V}\times\mathcal{V}\to\R$ is defined by
\begin{align*}
\mathscr{B}((u,\rho),(v,\phi)):=&\nu \int_\Omega\nabla u:\nabla v\,dx+\kappa\int_\Omega\nabla\rho\cdot\nabla\phi\,dx \\
&+\nu\int_{\partial\Omega}\alpha\,u_{\tanproj} \cdot{v}_{\tanproj}\,dS-\nu\int_{\partial\Omega}v\cdot d\nn(u)\,dS.
\end{align*}
Using Lemma~\ref{lem:boundary_identity}, an integration by parts shows that for smooth functions satisfying the boundary conditions, this form is equivalent to
\[
\mathscr{B}(({u},\rho),({v},\phi))=-\nu \int_\Omega v\cdot\Delta u\,dx-\kappa\int_\Omega\phi\Delta\rho\,dx.
\]

We now define the Boussinesq-Stokes operator $\boussop:\mathcal{V}\to\mathcal{V}'$ by
\[
\langle\boussop(u,\rho),(v,\phi)\rangle:=\mathscr{B}((u,\rho),(v,\phi)).
\]

\begin{proposition}[Properties of $\mathscr{B}$ and $\boussop$] \label{prop:bouss_op_props}
The bilinear form $\mathscr{B}$ is bounded and symmetric. Moreover, there exist constants $C_1,C_2>0$ and $\beta_0\ge0$ such that, for all $(u,\rho),(v,\phi) \in \mathcal{V}$:
\begin{align*}
|\mathscr{B}((u,\rho),(v,\phi))| &\leq C_1 \| (u,\rho) \|_{\mathcal{V}} \| (v, \phi) \|_{\mathcal{V}}, \\
\mathscr{B}((u,\rho),(v,\phi)) &= \mathscr{B}((v,\phi),(u,\rho)), \\
\mathscr{B}((u,\rho),(u,\rho)) &\geq C_2 \| (u,\rho) \|_{\mathcal{V}}^2 - \beta_0 \| (u,\rho) \|_{\mathcal{H}}^2.
\end{align*}
For any $\beta>\beta_0$, define
\[
\langle\widetilde{\boussop}_\beta U,V\rangle
:=\mathscr B(U,V)+\beta(U,V)_{\mathcal H}.
\]
Then $\widetilde{\boussop}_\beta$ is an isomorphism from $\mathcal{V}$ to $\mathcal{V}'$.
\end{proposition}

\begin{proof}
Boundedness follows from the trace theorem, $\alpha\in L^\infty(\partial\Omega)$, and $\|dn\|_{L^\infty(\partial\Omega)}\le\Lambda_\Omega$. Symmetry follows from self-adjointness of the shape operator on each tangent space. Finally, by the trace interpolation inequality and Young's inequality, we have 
\[
\mathscr B((u,\rho),(u,\rho))
\ge \frac{\nu}{2}\|\nabla u\|_{L^2}^2+\kappa\|\nabla\rho\|_{L^2}^2-C\|u\|_{L^2}^2.
\]
Thus the stated estimate holds with $C_2=\min\{\nu/2,\kappa\}$ and $\beta_0=C$. For any $\beta>\beta_0$, the shifted form is coercive, and the conclusion follows from the Lax--Milgram theorem.
\end{proof}

The final part of this section gives an equivalent form of the viscous term used in the energy inequalities.

\begin{proposition}[Extension of the viscous identity]\label{prop2.3}
Let $u,\Phi\in H^1_{\sigma,\tan}(\Omega)$. Then
\begin{equation*}
    \begin{split}
        \int_{\Omega} \nabla \Phi: \nabla u \,dx+\int_{\partial \Omega} \Phi \cdot[\alpha u &-d\boldsymbol{n}(u)] \,dS \\
        & =2\int_{\Omega} S(\Phi):S(u)\,dx+\int_{\partial \Omega} \alpha \Phi \cdot u\,dS.
    \end{split}
\end{equation*}
\end{proposition}
\begin{proof}
{Apply \cite[Proposition~4.3]{Kelliher2025} with zero
friction coefficient. For
\(\Phi\in C^\infty_{\sigma,\tan}(\overline\Omega)\), this gives}
\[
{
\int_\Omega\nabla\Phi:\nabla u\,dx
-\int_{\partial\Omega}\Phi\cdot d\boldsymbol n(u)\,dS
=
2\int_\Omega S(\Phi):S(u)\,dx.}
\]
{Adding
\(\int_{\partial\Omega}\alpha\,\Phi\cdot u\,dS\) to both sides yields the
claimed identity for smooth \(\Phi\). Since
\(\alpha\in L^\infty(\partial\Omega)\), the added boundary pairing is
continuous on \(H^1(\Omega)\times H^1(\Omega)\) by the trace theorem.
The result for arbitrary
\(\Phi\in H^1_{\sigma,\tan}(\Omega)\) follows from the
\(H^1\)-density of
\(C^\infty_{\sigma,\tan}(\overline\Omega)\). No approximation of
\(\alpha\) is required.}
\end{proof}

\section{Global Existence of Weak Solutions}\label{sec3}

We now define a weak solution.

\begin{defn}[Weak Solution]\label{def:weak_sol}
A pair $(u,\rho)$ is a weak solution to (\ref{eq:bouss_momentum})-(\ref{eq:navier_bc}) with initial data $(u_0,\rho_0)\in\mathcal{H}$ if:
\begin{itemize}
    \item $(u, \rho) \in L^2(0,T; \mathcal{V}) \cap C_w([0,T];\mathcal{H})$, where $C_w([0,T];\mathcal{H})$ denotes functions continuous from $[0,T]$ to $\mathcal{H}$ with the weak topology,
    \item $\partial_t (u,\rho) \in L^{4/3}(0,T; \mathcal{V}')$,
    \item For every test pair $(v,\phi)\in C_c^1([0,T);\mathcal V)$, the following identity holds:    \begin{equation}\label{weaksolution1}
    \begin{aligned}
    &-\int_0^T (u,\partial_t v)_{L^2}\,dt
    -\int_0^T (\rho,\partial_t\phi)_{L^2}\,dt \\
    &\quad+\int_0^T \mathscr{B}((u,\rho),(v,\phi))\,dt+\int_0^T\int_\Omega(u\cdot\nabla)u\cdot v\,dx\,dt+\int_0^T\int_\Omega (u\cdot\nabla\rho)\phi\,dx\,dt \\
    &=\int_0^T\int_\Omega \rho e_3\cdot v\,dx\,dt
    +(u_0,v(0))_{L^2}+(\rho_0,\phi(0))_{L^2}.
    \end{aligned}
    \end{equation}
\end{itemize}
\end{defn}

\begin{theorem}[Global existence of weak solutions] \label{thm:global_existence}
For any initial data $(u_0,\rho_0)\in\mathcal{H}$, there exists a global weak solution $(u,\rho)$ in the sense of Definition~\ref{def:weak_sol}.
\end{theorem}

\begin{proof}
The proof follows the classical Galerkin--Hopf construction; see also \cite{Kelliher2025} for a closely related argument. Define
\[
a_0(W,Z)
:=\int_\Omega\nabla W:\nabla Z\,dx
+\int_{\partial\Omega}Z\cdot[\alpha W-dn(W)]\,dS,
\qquad W,Z\in H^1_{\sigma,\tan}(\Omega),
\]
and, for a fixed $\beta>\beta_0/\nu$, set
\[
a_\beta(W,Z):=a_0(W,Z)+\beta(W,Z)_{L^2(\Omega)}.
\]
By Proposition~\ref{prop:bouss_op_props}, $a_\beta$ is coercive and defines an inner product equivalent to the $H^1$ inner product on $H^1_{\sigma,\tan}(\Omega)$. The compact embedding $H^1_{\sigma,\tan}(\Omega)\hookrightarrow L^2_{\sigma,\tan}(\Omega)$ and the spectral theorem give eigenvalues $\{\lambda_j\}_{j=1}^\infty$ and an $L^2$-orthonormal basis $\{v_j\}_{j=1}^\infty\subset H^1_{\sigma,\tan}(\Omega)$ such that
\[
a_\beta(v_j,Z)=\lambda_j(v_j,Z)_{L^2(\Omega)}
\qquad\text{for every }Z\in H^1_{\sigma,\tan}(\Omega).
\]
Writing $\mu_j:=\lambda_j-\beta$, we equivalently have
\[
a_0(v_j,Z)=\mu_j(v_j,Z)_{L^2(\Omega)}.
\]
This construction uses only $\alpha\in L^\infty(\partial\Omega)$. Similarly, let $\{\theta_j\}_{j=1}^\infty\subset H^1_0(\Omega)$ be an $L^2$-orthonormal basis of $L^2(\Omega)$ consisting of eigenfunctions of the Dirichlet Laplacian. For each $m$, set 
\[
X^m:=\operatorname{span}\{v_1,\dots,v_m\},
\qquad
Y^m:=\operatorname{span}\{\theta_1,\dots,\theta_m\}.
\]
We seek approximate solutions of the form
\[
u^m(t,x)=\sum_{j=1}^m g_j^m(t)v_j(x),
\qquad
\rho^m(t,x)=\sum_{j=1}^m h_j^m(t)\theta_j(x),
\]
with initial condition
\[
u^m(0)=P^m u_0,
\qquad
\rho^m(0)=Q^m\rho_0.
\]
Here $P^m$ and $Q^m$ denote the $L^2$-orthogonal projections onto $X^m$ and $Y^m$, respectively. Hence,
\[
(u^m(0),\rho^m(0))\to (u_0,\rho_0)\quad\text{strongly in }\mathcal H.
\]

\par\smallskip
For $j=1,\ldots,m$, testing the Galerkin formulation against $v_j$ and $\theta_j$, we obtain
\begin{equation}
    \begin{split}
        &\int_{\Omega}\partial_t u^m \cdot v_j\,dx
        +\int_{\Omega}\big[(u^m\cdot \nabla)u^m\big]\cdot v_j\,dx\\
        &\qquad=-\nu\bigg(\int_\Omega \nabla u^m:\nabla v_j\,dx+\int_{\partial\Omega}
        [\alpha u^m_{\tanproj}-dn(u^m)]\cdot v_j\,dS
        \bigg)+\int_\Omega \rho^m e_3\cdot v_j\,dx ,
    \end{split}
    \label{galerkin-u}
\end{equation}
and
\begin{equation}
\int_{\Omega}\partial_t \rho^m\,\theta_j\,dx+\int_{\Omega}(u^m\cdot\nabla\rho^m)\theta_j\,dx=-\kappa\int_\Omega\nabla\rho^m\cdot\nabla\theta_j\,dx.\label{galerkin-rho}
    \end{equation}
These equations yield a finite-dimensional system of ODEs for the unknown variables $\{g_j^m(t)\}_{j=1}^m$ and $\{h_j^m(t)\}_{j=1}^m$. Since the right-hand sides are continuous and locally Lipschitz, the classical ODE theory gives a local-in-time solution.

Let $t\in [0,T_m)$ on the maximal interval of existence for each $m$. A standard energy argument for \eqref{galerkin-u} and \eqref{galerkin-rho} gives the following estimate. Indeed, multiply \eqref{galerkin-u} by \(g_j^m(t)\),
multiply \eqref{galerkin-rho} by \(h_j^m(t)\), and sum over
\(j=1,\ldots,m\). The transport terms vanish because
\(\nabla\cdot u^m=0\), \(u^m\cdot n=0\), and \(\rho^m=0\) on
\(\partial\Omega\), while Young's inequality gives
\[
\left|\int_\Omega \rho^m e_3\cdot u^m\,dx\right|
\le \frac12\|\rho^m\|_{L^2}^2+\frac12\|u^m\|_{L^2}^2.
\]
Consequently,
\begin{multline*}
\frac{1}{2} \frac{d}{dt}
\left(
\|u^m\|_{L^2}^2+\|\rho^m\|_{L^2}^2
\right)
+\nu\|\nabla(u^m)\|_{L^2}^2
+\kappa\|\nabla\rho^m\|_{L^2}^2  \\
+\nu\int_{\partial\Omega}\alpha |u^m_{\tanproj}|^2\,dS
-\nu\int_{\partial\Omega}u^m\cdot dn(u^m)\,dS
\le \frac12\|\rho^m\|_{L^2}^2+\frac12\|u^m\|_{L^2}^2
\end{multline*}

Since $\Lambda_\Omega=\sup_{x\in\partial\Omega}\max_i|k_i(x)|$, we have
\[
\begin{aligned}
\int_{\partial\Omega}\left(u^m\cdot dn(u^m)-\alpha |u^m_{\tanproj}|^2\right)\,dS
&\le\int_{\partial\Omega}\left(\Lambda_\Omega-\alpha(x)\right)|u^m|^2\,dS \\
&\le\epsilon\|\nabla u^m\|^2_{L^2(\Omega)}+C_\epsilon\|u^m\|^2_{L^2(\Omega)},
\end{aligned}
\]
where we applied the trace and Young's inequality.
Choosing $\epsilon>0$ small enough and absorbing into dissipation yields 
\[
\frac{d}{dt}E^m(t)+c_0(\|\nabla u^m\|^2_{L^2}+\|\nabla\rho^m\|_{L^2}^2)\le C_0E^m(t)
\]
where
\[
E^m(t):=\|u^m(t)\|^2_{L^2}+\|\rho^m(t)\|^2_{L^2},
\qquad
c_0:=2\min\{\nu(1-\epsilon),\kappa\}.
\]
By the Gronwall inequality,
\begin{equation}\label{GronwallEm}
E^m(t)\le E^m(0)e^{C_0T}
\end{equation}
for any $t\le T$.
Since
\[
\|u^m(t)\|^2_{L^2(\Omega)}
=\sum_{j=1}^m |g_j^m(t)|^2,
\qquad
\|\rho^m(t)\|^2_{L^2(\Omega)}
=\sum_{j=1}^m |h_j^m(t)|^2,
\]
\eqref{GronwallEm} gives a uniform bound on any finite time interval for $E^m(t)$. Hence finite-time blow-up cannot occur, and the solution exists globally in time for each fixed $m$. Moreover, by integrating the differential inequality from $0$ to $t$, we obtain
\begin{align}\label{energybond}
\begin{split}
E^m(t)+c_0\int_0^t\left(\|\nabla u^m(s)\|^2_{L^2}+\|\nabla\rho^m(s)\|_{L^2}^2\right)\,ds
&\le C_0\int_0^tE^m(s)\,ds+E^m(0)\\
&\le E^m(0)e^{C_0t}\le E(0)e^{C_0T},
\end{split}
\end{align}
for every $t\le T$. Here we used the fact $P^m$ and $Q^m$ are $L^2$-orthogonal projections so
\[
E^m(0)=\|P^m u_0\|^2_{L^2}+\|Q^m\rho_0\|^2_{L^2}
\le \|u_0\|^2_{L^2}+\|\rho_0\|^2_{L^2}=E(0).
\]
This implies
\[
({u}^m,\rho^m)\in L^\infty(0,T;\mathcal H)\cap L^2(0,T;\mathcal V)
\]
for any $T>0$ uniformly in \(m\). 

The following estimates provide the uniform time-derivative bounds in
$L^{4/3}(0,T;\mathcal V')$ required by the Aubin--Lions lemma.
Let $(W,V)\in H^1_{\sigma,\tan}(\Omega)\times H^1_0(\Omega)$. Since $\partial_t u^m\in X^m$ and $\partial_t \rho^m\in Y^m$, we have
\[
\int_\Omega \partial_tu^m\cdot W\,dx
=
\int_\Omega \partial_tu^m\cdot P^mW\,dx,
\qquad
\int_\Omega \partial_t\rho^m V\,dx
=
\int_\Omega \partial_t\rho^m Q^mV\,dx.
\]
For the shifted form defined above, the eigenvalue relation for $a_0$ gives
\[
a_\beta(v_i,v_j)=(\mu_j+\beta)\delta_{ij},
\qquad
a_\beta(W,v_j)=(\mu_j+\beta)(W,v_j)_{L^2(\Omega)}.
\]
Consequently, the $L^2$ projection $P^m$ is also the $a_\beta$-orthogonal projection onto $X^m$. Since the $v_j$ are complete in $L^2_{\sigma,\tan}(\Omega)$, their span has trivial $a_\beta$-orthogonal complement and is therefore dense in the $a_\beta$-norm. Hence
\[
\|P^mW\|_{H^1(\Omega)}\le C\|W\|_{H^1(\Omega)},
\qquad
P^mW\to W\quad\text{strongly in }H^1_{\sigma,\tan}(\Omega).
\]
Similarly, the Dirichlet eigenfunctions are orthogonal for the $H^1_0$ inner product, so
\[
\|Q^mV\|_{H^1(\Omega)}\le C\|V\|_{H^1(\Omega)},
\qquad
Q^mV\to V\quad\text{strongly in }H^1_0(\Omega).
\]
Using \(P^mW\in X^m\) as a test function in
\eqref{galerkin-u}, we obtain
\begin{align*}
\begin{split}
\bigg|\int_{\Omega}\partial_tu^m\cdot W\,dx\bigg|&=\bigg|\int_{\Omega}\partial_tu^m\cdot P^mW\,dx\bigg|\\
&\le \bigg|\int_\Omega[(u^m\cdot\nabla)u^m]\cdot P^mW\,dx\bigg|
+\nu\bigg|\int_\Omega \nabla u^m:\nabla P^mW\,dx\bigg|\\
&+\nu\bigg|\int_{\partial\Omega}[\alpha u^m-d{n}(u^m)]\cdot P^mW\,dS\bigg|+\bigg|\int_\Omega\rho^m e_3\cdot P^mW\,dx\bigg|,
\end{split}
\end{align*}
Similarly, testing \eqref{galerkin-rho} with
\(Q^mV\in Y^m\) gives
\begin{align*}
\begin{split}
\bigg|\int_{\Omega}\partial_t\rho^m V\,dx\bigg|&=\bigg|\int_{\Omega}\partial_t\rho^m Q^mV\,dx\bigg|\\
&\le\bigg|\int_{\Omega}[(u^m\cdot\nabla)\rho^m] Q^mV\,dx\bigg|+\kappa\bigg|\int_{\Omega}\nabla\rho^m\cdot\nabla Q^mV\,dx\bigg|.
\end{split}
\end{align*}
Using the uniform \(H^1\)-stability of \(P^m\), the velocity diffusion term
satisfies
\[
\nu\bigg|\int_\Omega \nabla u^m:\nabla P^mW\,dx\bigg|
\le C\nu\|\nabla u^m\|_{L^2(\Omega)}\|W\|_{H^1(\Omega)}.
\]
For the velocity convection term, by
H\"older's inequality, interpolation, and
$H^1(\Omega)\hookrightarrow L^6(\Omega)$,
we have
\begin{align*}
\bigg|\int_\Omega[(u^m\cdot\nabla)u^m]\cdot P^mW\,dx\bigg|
&\le \|u^m\|_{L^3}\|\nabla u^m\|_{L^2}\|P^mW\|_{L^6}\\
&\le C\|u^m\|_{L^2}^{1/2}\|\nabla u^m\|_{L^2}^{3/2}\|W\|_{H^1}.
\end{align*}
By the trace theorem, the estimate
$|d\boldsymbol n(z)|\le\Lambda_\Omega|z|$ on $\partial\Omega$, and the uniform
$H^1$-stability of $P^m$, the boundary term satisfies
\begin{align*}
\nu\bigg|\int_{\partial\Omega}[\alpha u^m-d\boldsymbol n(u^m)]\cdot P^mW\,dS\bigg|
&\le C\nu\bigl(\|\alpha\|_{L^\infty(\partial\Omega)}+\Lambda_\Omega\bigr)
\|u^m\|_{H^1(\Omega)}\|W\|_{H^1(\Omega)}.
\end{align*}
By the Cauchy--Schwarz inequality and the uniform
\(H^1\)-boundedness of \(P^m\), the buoyancy term is bounded by
\[
\bigg|\int_\Omega\rho^m e_3\cdot P^mW\,dx\bigg|
\le C\|\rho^m\|_{L^2(\Omega)}\|W\|_{H^1(\Omega)}.
\]
For the scalar convection term, integration by parts, followed by H\"older's
inequality and Sobolev embedding, gives
\begin{align*}
\bigg|\int_\Omega (u^m\cdot\nabla)\rho^m\,Q^mV\,dx\bigg|
&=\bigg|\int_\Omega \rho^m u^m\cdot\nabla Q^mV\,dx\bigg|\\
&\le C\|u^m\|_{L^2}^{1/2}\|u^m\|_{H^1}^{1/2}
\|\rho^m\|_{H^1}\|V\|_{H^1}.
\end{align*}
Consequently,
\[
\|\partial_tu^m\|_{(H^1_{\sigma,\tan}(\Omega))'}
\le C\left(
\|u^m\|_{L^2}^{1/2}\|u^m\|_{H^1}^{3/2}
+\|u^m\|_{H^1}+\|\rho^m\|_{L^2}
\right),
\]
and
\[
\|\partial_t\rho^m\|_{(H^1_0(\Omega))'}
\le C\left(
\|u^m\|_{L^2}^{1/2}\|u^m\|_{H^1}^{1/2}\|\rho^m\|_{H^1}
+\|\rho^m\|_{H^1}
\right).
\]
It follows that
\begin{align*}
\int_0^T\|\partial_tu^m\|_{(H^1_{\sigma,\tan}(\Omega))'}^{4/3}\,ds
&\le C\|u^m\|_{L^\infty(0,T;L^2(\Omega))}^{2/3}
       \|u^m\|_{L^2(0,T;H^1(\Omega))}^2\\
&\quad+CT^{1/3}\|u^m\|_{L^2(0,T;H^1(\Omega))}^{4/3}
      +CT^{1/3}\|\rho^m\|_{L^2(0,T;L^2(\Omega))}^{4/3}\\
&\le C(E(0)e^{C_0T})^{4/3}
  +CT^{1/3}(E(0)e^{C_0T})^{2/3},
\end{align*}
where constants have been enlarged in the final line. Similarly,
\begin{align*}
\int_0^T\|\partial_t\rho^m\|_{(H^1_0(\Omega))'}^{4/3}\,ds
&\le C\|u^m\|_{L^\infty(0,T;L^2(\Omega))}^{2/3}
       \|u^m\|_{L^2(0,T;H^1(\Omega))}^{2/3}
       \|\rho^m\|_{L^2(0,T;H^1_0(\Omega))}^{4/3}\\
&\quad+CT^{1/3}\|\rho^m\|_{L^2(0,T;H^1_0(\Omega))}^{4/3}\\
&\le C(E(0)e^{C_0T})^{4/3}
  +CT^{1/3}(E(0)e^{C_0T})^{2/3}.
\end{align*}
Therefore,
\[
\|\partial_t u^m\|_{L^{4/3}(0,T;(H^1_{\sigma,\tan}(\Omega))')}
+\|\partial_t\rho^m\|_{L^{4/3}(0,T;(H^1_0(\Omega))')}
\le C_T
\]
for every $T>0$ and every $m\in\mathbb{N}$.

These estimates allow us to apply the Aubin--Lions lemma, which gives strong convergence of a subsequence of $(u^m,\rho^m)$ to $(u,\rho)$ in $L^2(0,T;\mathcal H)$. Combining this with the Banach--Alaoglu theorem and passing to a further subsequence, still indexed by \(m\), we have
\[
u^m\rightarrow u \text{\quad strongly in\quad} L^2(0,T;L^2(\Omega;\mathbb R^3)),
\]
\[
u^m\overset{\ast}{\rightharpoonup} u \text{\quad weakly-* in\quad} L^\infty(0,T;L^2_{\sigma,\tan}(\Omega)),
\]
\[
\nabla u^m\rightharpoonup \nabla u \text{\quad weakly in\quad} L^2(0,T;L^2(\Omega;\mathbb R^{3\times3})),
\]
\[
\rho^m\rightarrow \rho \text{\quad strongly in\quad} L^2(0,T;L^2),
\]
\[
\rho^m\overset{\ast}{\rightharpoonup} \rho\text{\quad weakly-* in\quad} L^\infty(0,T;L^2),
\]
\[
\nabla\rho^m\rightharpoonup \nabla\rho \text{\quad weakly in\quad} L^2(0,T;L^2(\Omega;\mathbb R^3)).
\]
A diagonal extraction in $M=1,2,\dots$ gives a subsequence, not relabeled,
along which all of these convergences hold on every finite time interval.

{For each \(M\in\mathbb N\), the uniform
time-derivative bounds and reflexivity give, after a further diagonal
extraction,}
\[
{
\partial_t(u^m,\rho^m)\rightharpoonup G_M
\quad\text{weakly in }L^{4/3}(0,M;\mathcal V')}
\]
{for some
\(G_M\in L^{4/3}(0,M;\mathcal V')\). For every
\(\Psi\in C_c^\infty(0,M;\mathcal V)\),}
\[
{
\int_0^M\langle G_M,\Psi\rangle\,dt
=
\lim_{m\to\infty}
\int_0^M\langle\partial_t(u^m,\rho^m),\Psi\rangle\,dt
=
-\int_0^M((u,\rho),\partial_t\Psi)_{\mathcal H}\,dt.}
\]
Hence
\(G_M=\partial_t(u,\rho)\) distributionally on \((0,M)\). Since \(M\) is
arbitrary,
\[
{
\partial_t(u,\rho)\in
L^{4/3}_{\mathrm{loc}}([0,\infty);\mathcal V').}
\]
Here the derivative is understood through the preceding
\(\mathcal V'\)-valued integration-by-parts identity.
After a further diagonal extraction, still along a subsequence not relabeled,
we also have
\[
{
(u^m(t),\rho^m(t))\longrightarrow(u(t),\rho(t))
\quad\text{strongly in }\mathcal H}
\]
{for almost every \(t>0\), while all the preceding convergences remain valid
on every finite time interval.}

Now fix $(v,\phi)\in C_c^\infty([0,T);C^\infty_{\sigma,\tan}(\overline{\Omega}))
\times C_c^\infty([0,T);C_0^\infty(\Omega))$. In the Galerkin equations we use the admissible test pair $(P^mv,Q^m\phi)$. Since the projections are uniformly bounded and converge strongly in the corresponding $H^1$ spaces, their convergence is uniform on the compact ranges of $v$, $\partial_tv$, $\phi$, and $\partial_t\phi$. Therefore
\[
P^mv\to v
\quad\text{strongly in }C^1([0,T];H^1_{\sigma,\tan}(\Omega)),
\qquad
Q^m\phi\to\phi
\quad\text{strongly in }C^1([0,T];H^1_0(\Omega)).
\]
Moreover, the uniform $H^1$ bounds on the projections and the standard trilinear estimates show that replacing $(P^m v,Q^m\phi)$ by $(v,\phi)$ in each limit introduces an error tending to zero. The strong \(L^2\) convergence gives the limits of the
time-derivative terms:
\[
\int_0^T\int_\Omega u^m\cdot\partial_tv\,dx\,dt\rightarrow\int_0^T\int_\Omega u\cdot\partial_tv\,dx\,dt\quad\text{ and }\quad \int_0^T\int_\Omega\rho^m\partial_t\phi \,dx\,dt\rightarrow\int_0^T\int_\Omega\rho\partial_t\phi\,dx\,dt.
\]
The buoyancy term converges by the strong \(L^2\) convergence of \(\rho^m\),
whereas the two diffusion terms converge by the weak \(L^2\) convergence of
\(\nabla u^m\) and \(\nabla\rho^m\), respectively:
\[
\int_0^T\int_\Omega\rho^m e_3\cdot v\,dx\,dt\rightarrow\int_0^T\int_\Omega\rho e_3\cdot v\,dx\,dt,
\]
\[
\int_0^T\int_\Omega \nabla u^m:\nabla v\,dx\,dt\rightarrow\int_0^T\int_\Omega \nabla u:\nabla v\,dx\,dt,
\]
and
\[
\int_0^T\int_\Omega \nabla\rho^m\cdot\nabla\phi\,dx\,dt\rightarrow\int_0^T\int_\Omega \nabla\rho\cdot\nabla\phi\,dx\,dt.
\]
Using incompressibility and the tangency condition, we write
\[
\int_0^T\int_\Omega [(u^m\cdot\nabla)u^m]\cdot v\,dx\,dt
=-\int_0^T\int_\Omega (u^m\otimes u^m):\nabla v\,dx\,dt.
\]
Since $u^m\to u$ strongly in $L^2(0,T;L^2(\Omega))$, we have
$u^m\otimes u^m\to u\otimes u$ strongly in $L^1((0,T)\times\Omega)$, and hence
\[
\int_0^T\int_\Omega [(u^m\cdot\nabla)u^m]\cdot v\,dx\,dt
\longrightarrow
\int_0^T\int_\Omega [(u\cdot\nabla)u]\cdot v\,dx\,dt.
\]
Likewise,
\[
\int_0^T\int_\Omega (u^m\cdot\nabla\rho^m)\phi\,dx\,dt
=-\int_0^T\int_\Omega \rho^m u^m\cdot\nabla\phi\,dx\,dt,
\]
and the strong $L^2$ convergence of both factors gives convergence to the corresponding limit term.
For the boundary terms, H\"older's inequality on
\((0,T)\times\partial\Omega\), the bound
\(|d\boldsymbol n(z)|\le\Lambda_\Omega|z|\), and the trace interpolation
inequality give
\begin{align}\label{boundest}
\begin{split}
    &\bigg|\int_0^T\int_{\partial\Omega} (\alpha{u}^m_{\tan}\cdot{v}_{\tan}-d{n}(u^m)\cdot {v})-(\alpha u_{\tan}\cdot v_{\tan}-d{n}(u)\cdot v)\,dS\,dt\bigg|\\
    \le &\|\alpha\|_{L^\infty(\partial \Omega)}\int_0^T\int_{\partial\Omega}|u^m_{\tan}-u_{\tan}|\cdot |{v}|\,dS\,dt+\int_0^T\int_{\partial\Omega}|d{n}(u^m-u)||v|\,dS\,dt\\
    \le &\|\alpha\|_{L^\infty(\partial \Omega)}\|u^m-u\|_{L^2(0,T;L^2(\partial\Omega))}\|v\|_{L^2(0,T;L^2(\partial\Omega))}\\
    &+\Lambda_\Omega\|u^m-u\|_{L^2(0,T;L^2(\partial\Omega))}\|v\|_{L^2(0,T;L^2(\partial\Omega))}\\
    \lesssim &\|\alpha\|_{L^\infty(\partial\Omega)}\|u^m-u\|_{L^2(0,T;H^1(\Omega))}^\frac{1}{2}\|u^m-u\|^\frac{1}{2}_{L^2(0,T;L^2(\Omega))}\|{v}\|_{L^2(0,T;H^1(\Omega))}\\
    &+\Lambda_\Omega\|u^m-u\|_{L^2(0,T;H^1(\Omega))}^\frac{1}{2}\|u^m-u\|_{L^2(0,T;L^2(\Omega))}^\frac{1}{2}\|v\|_{L^2(0,T;H^1(\Omega))}\\
    \lesssim &(\|\alpha\|_{L^\infty(\partial\Omega)}+\Lambda_\Omega)
    (\|u^m\|_{L^2(0,T;H^1(\Omega))}+\|u\|_{L^2(0,T;H^1(\Omega))})^\frac{1}{2}
    \|u^m-u\|_{L^2(0,T;L^2(\Omega))}^\frac{1}{2}\\
    &\qquad \times\|v\|_{L^2(0,T;H^1(\Omega))}
    \longrightarrow 0.
    \end{split}
\end{align}
Passing to the limit in the Galerkin formulation, we first obtain the weak formulation
for all smooth test pairs
\[
(v,\phi)\in C_c^\infty([0,T);C^\infty_{\sigma,\tan}(\overline{\Omega}))
\times C_c^\infty([0,T);C_0^\infty(\Omega)).
\]
By the density of $C^\infty_{\sigma,\tan}(\overline{\Omega})\times C_c^\infty(\Omega)$ in $\mathcal V$, together with standard approximation in time, this identity extends to every $(v,\phi)\in C_c^1([0,T);\mathcal V)$. Therefore the weak formulation holds for every test pair required in Definition~\ref{def:weak_sol}.
Since the bounds obtained above also give
\[
(u,\rho)\in L^\infty(0,T;\mathcal H)\cap L^2(0,T;\mathcal{V}),
\qquad
\partial_t(u,\rho)\in L^{4/3}(0,T;\mathcal{V}'),
\]
the standard Lions--Magenes weak-continuity theorem yields $(u,\rho)\in C_w([0,T];\mathcal H)$. We conclude that $(u,\rho)$ is a weak solution in the sense of Definition~\ref{def:weak_sol}.

\end{proof}

\section{Energy and Korn--Poincar\'e inequalities}\label{sec4}

The Galerkin-limit solution constructed in Theorem~\ref{thm:global_existence} satisfies the following energy inequalities.

\begin{lemma}[Energy inequalities]\label{lem:energy_inequalities}
Let $(u,\rho)$ be the weak solution constructed in Theorem~\ref{thm:global_existence}. There exists a full-measure set $\mathcal T\subset[0,\infty)$ with $0\in\mathcal T$ such that, for every $s\in\mathcal T$ and every $t\ge s$, the following inequalities hold:

\begin{align}
&\frac{1}{2}\|u(t)\|_{L^2}^2
  + \nu \int_s^t \|\nabla u\|_{L^2}^2\,d\tau
\nonumber \\
&\qquad \leq
  \frac{1}{2}\|u(s)\|_{L^2}^2
  +\nu\int_s^t \int_{\partial\Omega}
  \bigl(\Lambda_\Omega-\alpha\bigr)
  |u_{\tanproj}|^2\,dS\,d\tau
  +\int_s^t\int_\Omega \rho e_3\cdot u\,dx\,d\tau,\label{eq:energy_ineq1}
\\[0.5em]
&\frac{1}{2}\|u(t)\|_{L^2}^2
  +2\nu\int_s^t\|S(u)\|_{L^2}^2\,d\tau
\nonumber \\
&\qquad \leq
  \frac{1}{2}\|u(s)\|_{L^2}^2
  -\nu\int_s^t\int_{\partial\Omega}
  \alpha|u_{\tanproj}|^2 \, dS \, d\tau
  +\int_s^t\int_\Omega \rho e_3\cdot u\,dx\,d\tau,\label{eq:energy_ineq2}
\\[0.5em]
&\frac{1}{2}\|\rho(t)\|_{L^2}^2
  +\kappa\int_s^t\|\nabla\rho\|_{L^2}^2\,d\tau\leq\frac{1}{2}\|\rho(s)\|_{L^2}^2.\label{eq:rho_energy_ineq}
\end{align}
\end{lemma}
\begin{proof}
Testing the Galerkin momentum equation with $u^m$ gives the exact identity
\begin{equation}\label{energy-exact}
\frac12\frac{d}{dt}\|u^m\|_{L^2(\Omega)}^2
+\nu\left(\|\nabla u^m\|_{L^2(\Omega)}^2
+\int_{\partial\Omega}u^m\cdot(\alpha u^m-dn(u^m))\,dS\right)
=\int_\Omega\rho^m e_3\cdot u^m\,dx.
\end{equation}
Bounding only the shape-operator term gives
\begin{equation}\label{energy1}
\frac12\frac{d}{dt}\|u^m\|_{L^2(\Omega)}^2
+\nu\|\nabla u^m\|_{L^2(\Omega)}^2
\le \nu\int_{\partial\Omega}(\Lambda_\Omega-\alpha)|u^m|^2\,dS
+\int_\Omega\rho^m e_3\cdot u^m\,dx.
\end{equation}
On the other hand, Proposition~\ref{prop2.3} applied directly to \eqref{energy-exact} gives
\begin{equation}\label{energy-symmetric-exact}
\frac12\frac{d}{dt}\|u^m\|_{L^2(\Omega)}^2
+2\nu\|S(u^m)\|_{L^2(\Omega)}^2
+\nu\int_{\partial\Omega}\alpha|u^m|^2\,dS
=\int_\Omega\rho^m e_3\cdot u^m\,dx.
\end{equation}
We use the one-sided time-mollification and lower-semicontinuity
argument of \cite[Theorem~4.4]{Kelliher2025} to pass from
\eqref{energy1} to \eqref{eq:energy_ineq1}. Applied to
\eqref{energy-symmetric-exact}, the same argument yields
\eqref{eq:energy_ineq2}. Let $r\ge s$ and integrate in time from $s$ to $r$ in $\tau$ to obtain
    \begin{equation}\label{energy2}
    \begin{split}
        &\frac{1}{2}\|u^m(r,\cdot)\|^2_{L^2(\Omega)}+\nu\int_s^r\|\nabla u^m(\tau,\cdot)\|^2_{L^2(\Omega)}\,d\tau
        \\&
        \le \nu\int_s^r\int_{\partial\Omega}(\Lambda_\Omega-\alpha)|u^m|^2\,dS\,d\tau+\int_s^r\int_\Omega \rho^m e_3\cdot u^m\,dx\,d\tau
        +\frac{1}{2}\|u^m(s,\cdot)\|^2_{L^2(\Omega)}.
    \end{split}     
    \end{equation}
Fix $t>s$ and choose a finite horizon $T>t+1$. First suppose that $t$ is a Lebesgue point of the map $r\mapsto\|u(r)\|_{L^2(\Omega)}^2$, and choose $\psi\in C^\infty_c(\R_+)$ with $\int_{\R_+}|\psi|^2\,dr=1$ and $\operatorname{supp}\psi\subset(0,1)$. Let $0<\epsilon<1$ and set
    \[
    \psi_\epsilon(r):=\frac{1}{\sqrt\epsilon}\psi\left(\frac{r-t}{\epsilon}\right),
    \qquad r\in\R_+.
    \]
    We multiply \eqref{energy2} by $|\psi_\epsilon(r)|^2$. Since $t> s$, the function $\psi_\epsilon$ vanishes for $r\le s$, and hence we may integrate over $\R_+$ in $r$ to obtain
    \begin{align*}
    \begin{split}
        &\frac{1}{2}\int_0^\infty|\psi_\epsilon(r)|^2\|u^m(r,\cdot)\|^2_{L^2(\Omega)}\,dr+\nu\int_0^\infty|\psi_\epsilon(r)|^2\int_s^r\|\nabla u^m(\tau,\cdot)\|^2_{L^2(\Omega)}\,d\tau\,dr\\
        &\le \nu\int_0^\infty|\psi_\epsilon(r)|^2\int_s^r\int_{\partial\Omega}(\Lambda_\Omega-\alpha)|u^m|^2\,dS\,d\tau\,dr\\      &+\int_0^\infty|\psi_\epsilon(r)|^2\int_s^r\int_\Omega \rho^m e_3\cdot u^m\,dx\,d\tau\,dr+\frac{1}{2}\|u^m(s,\cdot)\|^2_{L^2(\Omega)}.
        \end{split}
    \end{align*}
    From Theorem~\ref{thm:global_existence}, we know that $u^m$ converges to $u$ weakly in $L^2(0,T;H^1_{\sigma,\tan}(\Omega))$ and strongly in $L^2(0,T;L^2(\Omega))$. Moreover, by the trace interpolation inequality,
    boundedness of $u^m$ in $L^2(0,T;H^1(\Omega))$ and
    strong convergence of $u^m\to u$ in $L^2(0,T;L^2(\Omega))$, it follows that
    $u^m\to u$ strongly in $L^2(0,T;L^2(\partial\Omega))$, as in \eqref{boundest}.
    Using Fubini's theorem and passing to the
    \(\liminf_{m\to\infty}\), we obtain
     \begin{equation}
     \begin{split}
        &\frac{1}{2}\int_0^\infty|\psi_\epsilon(r)|^2\|u(r,\cdot)\|^2_{L^2(\Omega)}\,dr+\nu\int_0^\infty|\psi_\epsilon(r)|^2\int_s^r\|\nabla u(\tau,\cdot)\|^2_{L^2(\Omega)}\,d\tau\,dr
        \\&\le \nu\int_0^\infty|\psi_\epsilon(r)|^2\int_s^r\int_{\partial\Omega}(\Lambda_\Omega-\alpha)|u|^2\,dS\,d\tau\,dr\\      &+\int_0^\infty|\psi_\epsilon(r)|^2\int_s^r\int_\Omega \rho e_3\cdot u\,dx\,d\tau \,dr+\frac{1}{2}\liminf_{m\rightarrow\infty}\|u^m(s,\cdot)\|^2_{L^2(\Omega)}.
    \end{split}
    \end{equation}
    Indeed, \(|\psi_\epsilon(r)|^2\,dr\) has unit mass and is
supported in \((t,t+\epsilon)\), so it is a one-sided approximate identity
at \(t\). The functions
\[
\begin{aligned}
r&\longmapsto\int_s^r\|\nabla u(\tau)\|_{L^2(\Omega)}^2\,d\tau,\\
r&\longmapsto\int_s^r\int_{\partial\Omega}
(\Lambda_\Omega-\alpha)|u|^2\,dS\,d\tau,\\
r&\longmapsto\int_s^r\int_\Omega \rho e_3\cdot u\,dx\,d\tau
\end{aligned}
\]
are absolutely continuous because their integrands belong
to \(L^1(s,T)\). Hence their averages converge to their values at
\(r=t\), while the averaged kinetic term converges to
\(\|u(t)\|_{L^2(\Omega)}^2\) by the Lebesgue-point assumption.
    \begin{equation}
    \begin{split}
         &\frac{1}{2}\|u(t,\cdot)\|^2_{L^2(\Omega)}+\nu\int_s^t\|\nabla u(\tau,\cdot)\|^2_{L^2(\Omega)}\,d\tau
         \\&\le \nu\int_s^t \int_{\partial\Omega}(\Lambda_\Omega-\alpha)|u|^2\,dS\,d\tau+\int_s^t \int_\Omega \rho e_3\cdot u\,dx\,d\tau\\
         &\quad+\frac{1}{2}\liminf_{m\rightarrow \infty}\|u^m(s,\cdot)\|^2_{L^2(\Omega)}.
    \end{split}
    \end{equation}
    The case $t=s$ is tautological. For an arbitrary $t>s$, choose Lebesgue
    points $t_n\downarrow t$. Weak continuity gives
    $u(t_n)\rightharpoonup u(t)$ in $L^2(\Omega)$, so the kinetic-energy term
    is lower semicontinuous, while all time-integral terms converge as
    $t_n\downarrow t$. The inequality therefore holds for every $t\ge s$.
    For almost every $s\in(0,T)$, the pointwise strong convergence gives
    \[
    \liminf_{m\to\infty}\|u^m(s)\|_{L^2(\Omega)}^2
    \le \|u(s)\|_{L^2(\Omega)}^2.
    \]
    Thus the desired energy inequality holds for almost every $s\in(0,T)$ and every $t\in[s,T]$.
    
    Applying the same limiting argument to \eqref{energy-symmetric-exact} gives \eqref{eq:energy_ineq2}. Testing the scalar Galerkin equation with $\rho^m$ gives \eqref{eq:rho_energy_ineq}. The transport term vanishes by incompressibility and $u^m\cdot n=0$.

    For each integer $M\ge1$, let $\mathcal T_M\subset[0,M]$ be the intersection of the three full-measure sets of initial times for which the corresponding inequalities hold for every $t\in[s,M]$. Set
    \[
    \mathcal T
    :=\{0\}\cup\bigcap_{M=1}^{\infty}\bigl(\mathcal T_M\cup(M,\infty)\bigr).
    \]
    Then $\mathcal T$ has full measure in $[0,\infty)$. The union with $(M,\infty)$ makes the condition from $\mathcal T_M$ automatic when $s>M$, so the countable intersection retains full measure on every finite interval. For $s=0$, the same mollifier argument applies, and the initial terms
    pass to the limit because
    \[
    u^m(0)\to u_0 \quad\text{strongly in }L^2(\Omega),
    \qquad
    \rho^m(0)\to \rho_0 \quad\text{strongly in }L^2(\Omega).
    \]
    Hence all three inequalities also hold at $s=0$.
    {If \(s\in\mathcal T\) and \(s>0\), choose an integer \(M\ge t\). Since
    \(s\le M\), membership in
    \(\mathcal T_M\cup(M,\infty)\) implies \(s\in\mathcal T_M\), and the
    inequalities hold at \((s,t)\).} Thus the same
    set $\mathcal T$ is admissible for all three energy inequalities.
\end{proof}

\medskip
We now define the corresponding energy class.
\begin{defn}[Leray--Hopf weak solution]\label{def:leray-hopf}
A weak solution $(u,\rho)$ in the sense of Definition~\ref{def:weak_sol} is called a Leray--Hopf weak solution if there exists a full-measure set $\mathcal T\subset[0,\infty)$, with $0\in\mathcal T$, such that the energy inequalities in Lemma~\ref{lem:energy_inequalities} hold for every $s\in\mathcal T$ and every $t\ge s$.
\end{defn}
Consequently, Theorem~\ref{thm:global_existence} and Lemma~\ref{lem:energy_inequalities} provide a global Leray--Hopf weak solution for every initial datum in $\mathcal H$.

We next state the Korn--Poincar\'e inequality needed for the proof of exponential decay. 

\begin{proposition}[Weighted Korn--Poincar\'e inequality]\label{thm:coupled_korn}
Let $\Omega\subset\mathbb{R}^3$ be a bounded, connected domain with smooth boundary, and let $\alpha\in L^\infty(\partial\Omega)$ satisfy $\alpha\ge0$ almost everywhere. Set $A_{\partial\Omega}:=\{x\in\partial\Omega:\alpha(x)>0\}$ and assume that $\mathcal H^2(A_{\partial\Omega})>0$. Then there exists a constant $C_{(KP,\alpha)}=C_{(KP,\alpha)}(\Omega,\alpha)>0$ such that 
\begin{equation}\label{kornpoincare2}
\|u\|^2_{L^2(\Omega)}\le C_{(KP,\alpha)}\left(\|S(u)\|_{L^2(\Omega)}^2+\int_{\partial\Omega}\alpha|u|^2\,dS\right)
\end{equation}
for every $u\in H^1_{\sigma,\tan}(\Omega)$.
In particular, taking $\alpha\equiv 1$, there exists $C_{KP}=C_{KP}(\Omega)>0$ such that
\begin{equation}\label{kornpoincare1}
\|u\|^2_{L^2(\Omega)}\le C_{KP}\left(\|S(u)\|_{L^2(\Omega)}^2+\int_{\partial\Omega}|u|^2\,dS\right).
\end{equation}
\end{proposition}
\begin{proof}
    By Proposition~\ref{prop2.3}, taking $\Phi=u$, we obtain
    \begin{equation}\label{korn2}
    \int_\Omega|\nabla u|^2\,dx=2\int_\Omega|S(u)|^2\,dx+\int_{\partial\Omega}u\cdot d{n}(u)\,dS.
    \end{equation}
    Since $\|d\boldsymbol n\|_{L^\infty(\partial\Omega)}\le\Lambda_\Omega$,  by H\"older inequality, 
    \[
    \int_{\partial\Omega}u\cdot d\boldsymbol n(u)\,dS
    \le \Lambda_\Omega\|u\|_{L^2(\partial\Omega)}^2.
    \]
    By the trace interpolation inequality, for every $\delta>0$,
    \[
    \|u\|_{L^2(\partial\Omega)}^2
    \le \delta\|\nabla u\|_{L^2(\Omega)}^2+C_\delta\|u\|_{L^2(\Omega)}^2.
    \]
    Combining this with \eqref{korn2} and choosing $\delta$ sufficiently small gives
    \begin{equation}\label{inequalitynablatoS}
    \|\nabla u\|^2_{L^2(\Omega)}\le 4\|S(u)\|^2_{L^2(\Omega)}+C_{\Omega,\Lambda_\Omega}\|u\|^2_{L^2(\Omega)}.
    \end{equation}
Estimate \eqref{inequalitynablatoS} also appears in
\cite[Proposition~5.1, Step~1]{Kelliher2025}.
    
    Now, we prove \eqref{kornpoincare2} by contradiction: suppose that no such constant exists. Then, for any $N\in\mathbb{N}$, there is $u_N\in H^1_{\sigma,\tan}(\Omega)$ such that 
    \begin{equation}
    \|u_N\|^2_{L^2(\Omega)}>N\left(\|S(u_N)\|_{L^2(\Omega)}^2+\int_{\partial\Omega}\alpha|u_N|^2\,dS\right).
    \end{equation}
    Without loss of generality, we assume $\|u_N\|_{L^2(\Omega)}=1$ for every $N$.
    Therefore,
    \begin{equation}\label{korn1}
    \frac{1}{N}>\|S(u_N)\|_{L^2(\Omega)}^2+\int_{\partial\Omega}\alpha|u_N|^2\,dS.
    \end{equation}
    Combining with \eqref{inequalitynablatoS}, we obtain the uniform bound of $\|\nabla u_N\|^2_{L^2(\Omega)}$.
    Hence, after passing to a subsequence, $u_N\rightharpoonup u'$ weakly in $H^1(\Omega)$ and $u_N\to u'$ strongly in $L^2(\Omega)$ by the Rellich--Kondrachov theorem.
    {The divergence operator and the normal trace are continuous
    linear maps on \(H^1(\Omega;\mathbb R^3)\), therefore
    \(H^1_{\sigma,\tan}(\Omega)\) is weakly closed in \(H^1(\Omega;\mathbb
    R^3)\), and \(u'\in H^1_{\sigma,\tan}(\Omega)\).}
    We also get $\|u'\|_{L^2(\Omega)}=1$. By compactness of the trace map, we also have $u_N\rightarrow u'$ strongly in $L^2(\partial\Omega)$. Moreover, since $S$ is a bounded linear operator, $S(u_N)\rightharpoonup S(u')$ weakly in $L^2(\Omega)$. On the other hand, \eqref{korn1} gives $\|S(u_N)\|_{L^2(\Omega)}^2+\int_{\partial\Omega}\alpha|u_N|^2\,dS\to0$. Hence $S(u')=0$, and strong convergence of the traces, together with $\alpha\in L^\infty(\partial\Omega)$, yields $\int_{\partial\Omega}\alpha|u'|^2\,dS=0$. Since $\alpha\ge0$, it follows that $u'=0$ almost everywhere on $A_{\partial\Omega}$.
    
    We now use the standard characterization that every $w\in \kernel S$ has the form $w=a+b\times x$ for some fixed $a,b\in\mathbb{R}^3$,  by applying Section~7 in \cite{Kelliher2025}. If $b=0$, then $w=0$ on a set of positive surface measure implies $a=0$. If $b\neq 0$, the zero set of $a+b\times x$ is either empty or an affine line parallel to $b$, and therefore has zero two-dimensional surface measure. Since $u'=0$ a.e.\ on $A_{\partial\Omega}$, which has positive surface measure, we conclude that $a=b=0$. Hence $u'=0$ in $\Omega$, contradicting $\|u'\|_{L^2(\Omega)}=1$. Taking $\alpha\equiv1$ gives the unweighted estimate.
\end{proof}
\begin{remark}
For comparison, \cite[Proposition~5.1]{Kelliher2025} proves that, for $u\in (\kernel S)^\perp:=\{u\in H^1_{\sigma,\tan}(\Omega):\int_\Omega u\cdot v\,dx=0$ for all $v\in\kernel S\}$,
\begin{equation}\label{kornpoincare3}
\|u\|^2_{L^2(\Omega)}\le C_\perp\|S(u)\|^2_{L^2(\Omega)}.
\end{equation}
Writing $u=w+u_\perp$, where $w=\Proj_{\kernel S}u$ and
$u_\perp\in(\kernel S)^\perp$, we have
$\|u\|_{L^2}^2=\|w\|_{L^2}^2+\|u_\perp\|_{L^2}^2$ and
$S(u)=S(u_\perp)$. Thus \eqref{kornpoincare3} controls the component
orthogonal to the entire kernel. Proposition~\ref{thm:coupled_korn} additionally
controls the kernel component through the weighted boundary term whenever
$\mathcal H^2(A_{\partial\Omega})>0$. When $\alpha\equiv0$ almost everywhere,
that boundary control disappears.
\end{remark}

\section{Exponential decay when $\alpha$ is positive on a set of positive surface measure}\label{sec5}

We now prove exponential decay when the nonnegative friction coefficient is positive on a boundary set of positive surface measure.

\begin{theorem}[Exponential decay] \label{thm:main}
Let $(u,\rho)$ be a global Leray--Hopf weak solution to the Boussinesq system \eqref{eq:bouss_momentum}--\eqref{eq:navier_bc} with initial data $(u_0,\rho_0)\in\mathcal H$. Set $A_{\partial\Omega}:=\{x\in\partial\Omega:\alpha(x)>0\}$ and assume that $\mathcal H^2(A_{\partial\Omega})>0$. Then there exist constants $K,D>0$, depending only on $\nu$, $\kappa$, $\Omega$, and $\alpha$, such that 
\[
\| {u}(t)\|_{L^2}^2+\|\rho(t)\|_{L^2}^2\leq D\left(\|{u}_0\|_{L^2}^2+\|\rho_0 \|_{L^2}^2\right) e^{-Kt}.
\]
\end{theorem}

\begin{proof}
Let $\mathcal T$ be the full-measure set in the Leray--Hopf energy inequalities. The Poincar\'e inequality gives a constant $C_P=C_P(\Omega)>0$ such that, for almost every time,
\begin{equation}
    \begin{split}
        \|\rho\|^2_{L^2(\Omega)}\le C_P\|\nabla\rho\|^2_{L^2(\Omega)}.
    \end{split}
\end{equation} 
Applying it to \eqref{eq:rho_energy_ineq}, we obtain 
\begin{equation}
   \begin{split}
      \|\rho(t)\|^2_{L^2}+\frac{2\kappa}{C_P}\int_s^t\|\rho(\tau)\|^2_{L^2}d\tau\le \|\rho(s)\|^2_{L^2}
   \end{split}
\end{equation}
for every $s\in\mathcal T$ and every $t\ge s$. 
Applying Corollary~\ref{Acoro} on this same set $\mathcal T$, we obtain
\begin{equation}\label{rho1}
    \|\rho(t)\|^2_{L^2(\Omega)}\le\|\rho(s)\|^2_{L^2(\Omega)}e^{-\frac{2\kappa}{C_P}(t-s)}.
\end{equation}
In particular, taking $s=0$ gives
\begin{equation}\label{rho2}
    \|\rho(t)\|^2_{L^2(\Omega)}\le\|\rho(0)\|^2_{L^2(\Omega)}e^{-\frac{2\kappa}{C_P}t}.
\end{equation}

We next estimate the velocity. Applying \eqref{eq:energy_ineq2} for every $s\in\mathcal T$ and every $t\ge s$, \eqref{kornpoincare2}, Young's inequality, and \eqref{rho1}, we obtain 
\begin{equation}\label{energyu2}
\begin{split}
\|u(t)\|^2_{L^2(\Omega)}+&(\frac{2\nu}{C_{(KP,\alpha)}}-\epsilon)\int_s^t\|u(\tau)\|^2_{L^2(\Omega)}\,d\tau 
\\&\le \frac{1}{\epsilon}\int_s^t\|\rho(\tau)\|^2_{L^2{(\Omega)}}\,d\tau+\|u(s)\|^2_{L^2(\Omega)}\\& \le\frac{1}{\epsilon}\|\rho(s)\|^2_{L^2(\Omega)}\int_s^te^{-\frac{2\kappa}{C_P}(\tau-s)}\,d\tau+\|u(s)\|^2_{L^2(\Omega)}.
\end{split}
\end{equation}
Choose $\epsilon>0$ such that $\frac{2\nu}{C_{(KP,\alpha)}}-\epsilon>0$ and $\frac{2\nu}{C_{(KP,\alpha)}}-\epsilon\neq \frac{2\kappa}{C_P}$. Equations~\eqref{rho1} and \eqref{energyu2} verify, on the common set $\mathcal T$, the two hypotheses of Proposition~\ref{AppendixA}.  Applying that proposition yields
\begin{equation}\label{u0}
\begin{split}
    &\|u(t)\|^2_{L^2(\Omega)}\le e^{-(\frac{2\nu}{C_{(KP,\alpha)}}-\epsilon)(t-s)}\|u(s)\|^2_{L^2(\Omega)}\\
    +&\frac{1}{\epsilon}\|\rho(s)\|^2_{L^2(\Omega)}\big[\frac{e^{-\frac{2\kappa}{C_P}(t-s)}-e^{-(\frac{2\nu}{C_{(KP,\alpha)}}-\epsilon)(t-s)}}{\frac{2\nu}{C_{(KP,\alpha)}}-\epsilon-\frac{2\kappa}{C_P}}\big],
\end{split}
\end{equation}
and
\begin{equation}\label{u1}
    \|u(t)\|^2_{L^2(\Omega)}\le e^{-(\frac{2\nu}{C_{(KP,\alpha)}}-\epsilon)t}\|u(0)\|^2_{L^2(\Omega)}+\frac{1}{\epsilon}\|\rho(0)\|^2_{L^2(\Omega)}\big[\frac{e^{-\frac{2\kappa}{C_P}t}-e^{-(\frac{2\nu}{C_{(KP,\alpha)}}-\epsilon)t}}{\frac{2\nu}{C_{(KP,\alpha)}}-\epsilon-\frac{2\kappa}{C_P}}\big]
\end{equation}
when $s=0$.
Set
\[
K:=\min\left\{\frac{2\nu}{C_{(KP,\alpha)}}-\epsilon,\frac{2\kappa}{C_P}\right\},
\qquad
D:=1+\frac{1}{\epsilon\left|\frac{2\nu}{C_{(KP,\alpha)}}-\epsilon-\frac{2\kappa}{C_P}\right|}.
\]
Combining \eqref{rho2} and \eqref{u1}, we obtain
\[
\|u(t)\|_{L^2(\Omega)}^2+\|\rho(t)\|_{L^2(\Omega)}^2
\le
D\bigl(\|u_0\|_{L^2(\Omega)}^2+\|\rho_0\|_{L^2(\Omega)}^2\bigr)e^{-Kt}.
\]
This proves the result.
\end{proof}

For the following corollary, a Leray--Hopf weak solution of the Navier--Stokes equations with Navier boundary conditions means a velocity field $u$ such that $(u,0)$ is a Leray--Hopf weak solution in the sense of Definition~\ref{def:leray-hopf}.

\begin{corollary}[Navier--Stokes decay under localized positive friction]\label{cor:NS-decay}
Let $\Omega\subset\mathbb R^3$ be bounded and connected with smooth boundary, and let $\alpha\in L^\infty(\partial\Omega)$ satisfy $\alpha\ge0$ almost everywhere and
\[
\mathcal H^2\bigl(\{x\in\partial\Omega:\alpha(x)>0\}\bigr)>0.
\]
Let $u$ be a Leray--Hopf weak solution of the three-dimensional incompressible Navier--Stokes equations with Navier boundary conditions and initial datum $u_0\in L^2_{\sigma,\tan}(\Omega)$. Then
\[
\|u(t)\|_{L^2(\Omega)}^2
\le \|u_0\|_{L^2(\Omega)}^2
\exp\left(-\frac{2\nu}{C_{(KP,\alpha)}}t\right),
\qquad t\ge0.
\]
\end{corollary}
\begin{proof}
By Definition~\ref{def:leray-hopf}, there exists a full-measure set $\mathcal T\subset[0,\infty)$ with $0\in\mathcal T$ such that the second velocity energy inequality with $\rho\equiv0$, together with \eqref{kornpoincare2}, gives
\[
\|u(t)\|_{L^2}^2+\frac{2\nu}{C_{(KP,\alpha)}}\int_s^t\|u(\tau)\|_{L^2}^2\,d\tau
\le \|u(s)\|_{L^2}^2
\]
for every $s\in\mathcal T$ and every $t\ge s$. Corollary~\ref{Acoro} gives the conclusion.
\end{proof}

\section{Exponential decay when $\alpha\equiv 0$}\label{sec6}
In this section, we address the decay of solutions when the friction coefficient $\alpha\equiv 0$ a.e.\ on the boundary. Let $\Proj_{\kernel S}$ denote the $L^2(\Omega)$-orthogonal projection onto $\kernel S$.

\begin{theorem}[Exponential decay for $\alpha\equiv0$]
\label{thm:decay_alpha_zero}
Let $(u,\rho)$ be a Leray--Hopf weak solution of \eqref{eq:bouss_momentum}--\eqref{eq:navier_bc} with $\alpha\equiv0$ almost everywhere on $\partial\Omega$ and initial data $(u_0,\rho_0)\in\mathcal H$. Then there exist constants $K,D>0$, depending only on $\nu$, $\kappa$, and $\Omega$, such that
\[
\|u(t)-\Proj_{\kernel S}u(t)\|_{L^2}^2+\|\rho(t)\|_{L^2}^2
\le D\left(\|u_0-\Proj_{\kernel S}u_0\|_{L^2}^2+\|\rho_0\|_{L^2}^2\right)e^{-Kt}.
\]
In particular, if $\kernel S=\{0\}$, the same estimate holds for the full velocity $u$.
\end{theorem}

\begin{proof}
Let $\mathcal T$ be the full-measure set in the Leray--Hopf energy inequalities, and let $C_P=C_P(\Omega)$ be the Poincar\'e constant for $H^1_0(\Omega)$. Since $S(w)=0$ implies that $w$ is a rigid motion, and since the tangency condition $w\cdot n=0$ restricts such rigid motions to the rotational symmetries of $\Omega$, it follows from \cite[Lemma~7.1, Proposition~7.2, and Corollary~7.3]{Kelliher2025} that $\kernel S$ is finite-dimensional and every element of $\kernel S$ is a smooth rigid motion. Let $N:=\dim\kernel S$, possibly $N=0$. Choose an $L^2$-orthonormal basis $\{w_j\}_{j=1}^N$ of $\kernel S$; when $N=0$, the sums below are empty. For $j=1,\dots,N$, define
\[
a_j(t):=(u(t),w_j)_{L^2(\Omega)}.
\]
Since $u\in C_w([0,\infty);L^2_{\sigma,\tan}(\Omega))$, each $a_j$ is continuous. We write
\[
u(t,x)=w(t,x)+u_\perp(t,x),
\qquad
w(t,x)=\sum_{j=1}^N a_j(t)w_j(x),
\]
where $w\in\kernel S$ and $u_\perp\in(\kernel S)^\perp$.
By letting $\eta\in C^1_c((0,\infty))$ and using the test function $(v,\phi)=(\eta(\tau)w_j,0)$ in the weak formulation, we obtain the following:
\begin{align*}
&-\int_0^\infty a_j(t)\eta'(t)\,dt+\int_0^\infty\eta(t)\mathscr{B}((u,\rho),(w_j,0))\,dt+\int_0^\infty\eta(t)\int_\Omega(u\cdot\nabla)u\cdot w_j\,dx\,dt\\
&=\int_0^\infty\eta(t)\int_\Omega\rho e_3\cdot w_j\,dx\,dt.
\end{align*}
We claim that the two middle terms vanish. Since $\alpha=0$ and $S(w_j)=0$, Proposition~\ref{prop2.3} gives
\[
\mathscr{B}((u,\rho),(w_j,0))=2\nu\int_\Omega S(u):S(w_j)\,dx=0.
\]
By using $\nabla\cdot u=0$ and $u\cdot n=0$ on $\partial\Omega$, we have
\begin{align*}
\int_\Omega[(u\cdot\nabla)u]\cdot w_j\,dx=-\int_\Omega[(u\cdot\nabla)w_j]\cdot u\,dx=-\int_\Omega u_iu_k\partial_i(w_j)_k\,dx.
\end{align*}
Since \(u_i u_k\) is symmetric in the indices \(i,k\), while \(S(w_j)=0\) implies that
\(\partial_i(w_j)_k\) is skew-symmetric in \(i,k\), we have
\[
u_i u_k\,\partial_i(w_j)_k
=
\frac12 u_i u_k
\left(\partial_i(w_j)_k+\partial_k(w_j)_i\right)
=0.
\]
Therefore
\[
\int_\Omega[(u\cdot\nabla)u]\cdot w_j\,dx=0.
\]
Hence, in the sense of distributions in time, 
\begin{equation}\label{aj-ODE}
\frac{d}{d\tau} a_j(\tau)=\int_\Omega \rho(\tau,x)\, e_3\cdot w_j(x)\,dx,\qquad j=1,\dots,N.
\end{equation}
The right-hand side is locally bounded in time by the scalar energy inequality. Hence $a_j$ agrees on $(0,T)$ with an absolutely continuous function. This function has a finite limit as $t\downarrow0$, and the continuity of $a_j(t)=(u(t),w_j)_{L^2}$ at $t=0$ identifies that limit with $a_j(0)$. Therefore
\[
a_j(t)=a_j(0)+\int_0^t\int_\Omega\rho(\tau,x)e_3\cdot w_j(x)\,dx\,d\tau.
\]
The chain rule therefore gives, for every $0\le s\le t$,
\[
a_j(t)^2-a_j(s)^2=2\int_s^t\int_\Omega a_j(\tau)\rho(\tau,x)e_3\cdot w_j(x)\,dx\,d\tau.
\]
Therefore
\begin{align}
\|w(t)\|_{L^2(\Omega)}^2
&=\sum_{j=1}^N a_j(t)^2 \notag\\
&=\sum_{j=1}^N a_j(s)^2
+2\sum_{j=1}^N\int_s^t a_j(\tau)
\int_\Omega \rho(\tau,x)e_3\cdot w_j(x)\,dx\,d\tau \notag\\
&=\|w(s)\|_{L^2(\Omega)}^2
+2\int_s^t\int_\Omega \rho(\tau,x)e_3\cdot w(\tau,x)\,dx\,d\tau.
\label{aj-ODE2}
\end{align}

For every $s\in\mathcal T$ and $t\ge s$, \eqref{eq:energy_ineq2} with $\alpha=0$ gives
\[
\|u(t)\|^2_{L^2(\Omega)}+4\nu\int_s^t\|S(u(\tau))\|^2_{L^2(\Omega)}\,d\tau\le\|u(s)\|^2_{L^2(\Omega)}+2\int_s^t\int_\Omega\rho e_3\cdot u\,dx\,d\tau.
\]
Since $u=u_\perp+w$, and $S(w)=0$, we have
\[
\|S(u(t))\|^2_{L^2(\Omega)}=\|S(u_\perp(t))\|^2_{L^2(\Omega)}
\]
and 
\begin{equation}\label{full-energy-integrated}
\begin{aligned}
&\|u_\perp(t)\|^2_{L^2(\Omega)}
+\|w(t)\|^2_{L^2(\Omega)}
+4\nu\int_s^t \|S(u_\perp(\tau))\|_{L^2(\Omega)}^2\,d\tau  \\
&\qquad\le
\|u_\perp(s)\|^2_{L^2(\Omega)}
+\|w(s)\|^2_{L^2(\Omega)}
+2\int_s^t\int_\Omega \rho(\tau,x)e_3\cdot u_\perp(\tau,x)\,dx\,d\tau \\
&\qquad\quad
+2\int_s^t\int_\Omega \rho(\tau,x)e_3\cdot w(\tau,x)\,dx\,d\tau.
\end{aligned}
\end{equation}
Subtracting the identity \eqref{aj-ODE2} from \eqref{full-energy-integrated}, we obtain 
\begin{equation}\label{uperp-energy-integrated}
\|u_\perp(t)\|^2_{L^2(\Omega)}
+4\nu\int_s^t \|S(u_\perp(\tau))\|^2_{L^2(\Omega)}\,d\tau
\le
\|u_\perp(s)\|^2_{L^2(\Omega)}
+
2\int_s^t\int_\Omega \rho(\tau,x)e_3\cdot u_\perp(\tau,x)\,dx\,d\tau.
\end{equation}
Since $u_\perp\in (\kernel S)^\perp$, by \eqref{kornpoincare3} and Young's inequality, we obtain
\begin{equation}\label{uperp-integrated}
\|u_\perp(t)\|^2_{L^2(\Omega)}
+
\left(\frac{4\nu}{C_\perp}-\epsilon\right)
\int_s^t \|u_\perp(\tau)\|^2_{L^2(\Omega)}\,d\tau
\le
\|u_\perp(s)\|^2_{L^2(\Omega)}
+
\frac{1}{\epsilon}
\int_s^t \|\rho(\tau)\|^2_{L^2(\Omega)}\,d\tau.
\end{equation}
The scalar energy estimate and the Poincar\'e inequality give
\begin{equation*}
    \|\rho(t)\|^2_{L^2(\Omega)}\le\|\rho(s)\|^2_{L^2(\Omega)}e^{-\frac{2\kappa}{C_P}(t-s)}.
\end{equation*}
The scalar estimate and \eqref{uperp-integrated} hold on the same full-measure set $\mathcal T$; hence they verify the hypotheses of Proposition~\ref{AppendixA}. Choose $\epsilon>0$ so that $\frac{4\nu}{C_\perp}-\epsilon>0$ and $\frac{4\nu}{C_\perp}-\epsilon\neq \frac{2\kappa}{C_P}$. Applying the proposition gives
\begin{equation}
\begin{split}
    &\|u_\perp(t)\|^2_{L^2(\Omega)}\le e^{-(\frac{4\nu}{C_\perp}-\epsilon)(t-s)}\|u_\perp(s)\|^2_{L^2(\Omega)}\\
    +&\frac{1}{\epsilon}\|\rho(s)\|^2_{L^2(\Omega)}\big[\frac{e^{-\frac{2\kappa}{C_P}(t-s)}-e^{-(\frac{4\nu}{C_\perp}-\epsilon)(t-s)}}{\frac{4\nu}{C_\perp}-\epsilon-\frac{2\kappa}{C_P}}\big],
\end{split}
\end{equation}
and 
\begin{equation}
    \|u_\perp(t)\|^2_{L^2(\Omega)}\le e^{-(\frac{4\nu}{C_\perp}-\epsilon)t}\|u_\perp(0)\|^2_{L^2(\Omega)}+\frac{1}{\epsilon}\|\rho(0)\|^2_{L^2(\Omega)}\big[\frac{e^{-\frac{2\kappa}{C_P}t}-e^{-(\frac{4\nu}{C_\perp}-\epsilon)t}}{\frac{4\nu}{C_\perp}-\epsilon-\frac{2\kappa}{C_P}}\big]
\end{equation}
when $s=0$.
Therefore there exist constants $D>0$ and $K>0$ such that
\[
\|u_\perp(t)\|_{L^2(\Omega)}^2+\|\rho(t)\|_{L^2(\Omega)}^2\le D(\|u_0-\Proj_{\kernel S}u_0\|^2_{L^2(\Omega)}+\|\rho_0\|^2_{L^2(\Omega)})e^{-Kt},
\qquad t\ge0.
\]
This completes the proof.
\end{proof}

\begin{remark}
    When $\alpha\equiv0$ and $\kernel S\neq\{0\}$, the dissipation does not control the kernel component $w$, because the weighted boundary term vanishes identically. Therefore, the full velocity need not decay to zero.
    More precisely, recall
    \[
        w(t,x)=\sum_{j=1}^N a_j(t)w_j(x),
    \]
    and
    \[
        \frac{d}{dt}a_j(t)
        =
        \int_\Omega \rho(t,x)\, e_3\cdot w_j(x)\,dx,
        \qquad j=1,\dots,N.
    \]
    The scalar decay estimate implies that
    \[
        a_{j,\infty}
        :=
        a_j(0)+
        \int_0^\infty\int_\Omega \rho(\tau,x)\, e_3\cdot w_j(x)\,dx\,d\tau
    \]
    is well defined and that $a_j(t)\to a_{j,\infty}$.
    Moreover,
    \[
    |a_j(t)-a_{j,\infty}|
    \le C_j\int_t^\infty\|\rho(\tau)\|_{L^2}\,d\tau
    \le C_j' e^{-\kappa t/C_P},
    \]
    so the convergence of the kernel component is exponential. The limit $a_{j,\infty}$ need not vanish. For example, if $\rho_0=0$, then $\rho\equiv0$ and $a_{j,\infty}=a_j(0)$ for every $j$; in particular, at least one limiting coefficient is nonzero whenever $\Proj_{\kernel S}u_0\neq0$. This recovers the Navier--Stokes behavior described in
{\cite[Proposition~7.4 and Theorem~7.5]{Kelliher2025}}; the identity for $a_j'$ above shows precisely how buoyancy breaks conservation of the kernel projection. Thus the orthogonal component decays exponentially to zero, while the full velocity converges exponentially to
    \[
        u(t)\to w_\infty
        :=
        \sum_{j=1}^N a_{j,\infty}w_j
    \]
    strongly in $L^2(\Omega)$ as $t\to\infty$.
\end{remark}

\section*{Acknowledgments}
WF was partially supported by the AMS-Simons Research Enhancement Grant for PUI Faculty. WW was partially supported by a Simons Foundation grant (No. 0007730).

\appendix
\section{A Two-Time Gronwall-Type Inequality}\label{appx1}
In this appendix, we prove a two-time Gronwall-type inequality with exponentially decaying forcing, inspired by Appendix A of \cite{Kelliher2025}. We do not impose a sign condition on $y$, because the proof uses a shifted quantity that need not have a fixed sign.

\begin{proposition}\label{AppendixA}
Let {$y$ be a fixed measurable representative of an
element of \(L^1_{\mathrm{loc}}([0,\infty))\)}, let
$q:[0,\infty)\to[0,\infty)$, and let $K,D',C'>0$ with $K\neq D'$.
{All pointwise inequalities below are understood for this
fixed representative.} Assume that there exists a full-measure set $\mathcal T\subset[0,\infty)$ with $0\in\mathcal T$ such that, for every $s\in\mathcal T$ and every $t\ge s$,
\begin{equation}\label{Aq}
q(t)\le q(s)e^{-D'(t-s)}
\end{equation}
and
\begin{equation}\label{A1}
y(t)+K\int_s^t y(\tau)\,d\tau
\le y(s)+C'q(s)\int_s^t e^{-D'(\tau-s)}\,d\tau.
\end{equation}
Then, for every $s\in\mathcal T$ and every $t\ge s$,
\begin{equation}\label{A2}
y(t)\le y(s)e^{-K(t-s)}+\frac{C'q(s)}{K-D'}\left(e^{-D'(t-s)}-e^{-K(t-s)}\right).
\end{equation}
No sign assumption on $y$ is required.
\end{proposition}

\begin{proof}
Fix $s\in\mathcal T$ and define
\[
B_s:=\frac{C'q(s)}{K-D'},
\qquad
z_s(t):=y(t)-B_se^{-D'(t-s)},
\qquad t\ge s.
\]
The constant $B_s$ may be negative. Let $r\in\mathcal T\cap[s,t]$. Applying \eqref{A1} with initial time $r$ and using \eqref{Aq}, we obtain
\[
\begin{aligned}
y(t)+K\int_r^t y(\tau)\,d\tau
&\le y(r)+C'q(r)\int_r^t e^{-D'(\tau-r)}\,d\tau\\
&\le y(r)+C'q(s)\int_r^t e^{-D'(\tau-s)}\,d\tau.
\end{aligned}
\]
A direct calculation, using
\[
B_s+\frac{C'q(s)-KB_s}{D'}=0,
\]
gives
\begin{equation}\label{Az}
z_s(t)+K\int_r^t z_s(\tau)\,d\tau\le z_s(r).
\end{equation}
Set
\[
Z_s(t):=\int_s^t z_s(\tau)\,d\tau,
\qquad
g_s(t):=z_s(t)+KZ_s(t).
\]
For almost every $r\in[s,t]$, \eqref{Az} gives $g_s(t)\le g_s(r)$, and taking $r=s$ gives $g_s(t)\le g_s(s)=z_s(s)$. Since $Z_s'(t)=z_s(t)$ almost everywhere,
\[
Z_s'(t)+KZ_s(t)=g_s(t),
\qquad Z_s(s)=0,
\]
so
\[
Z_s(t)=\int_s^t e^{-K(t-r)}g_s(r)\,dr.
\]
Because $g_s(r)\ge g_s(t)$ for almost every $r\in[s,t]$,
\[
\begin{aligned}
z_s(t)
&=g_s(t)-KZ_s(t)\\
&\le g_s(t)-Kg_s(t)\int_s^t e^{-K(t-r)}\,dr\\
&=e^{-K(t-s)}g_s(t)
\le e^{-K(t-s)}z_s(s).
\end{aligned}
\]
Returning to $y$ yields \eqref{A2}.
\end{proof}

\begin{corollary}\label{Acoro}
Let {$y$ be a fixed measurable representative of an element
of \(L^1_{\mathrm{loc}}([0,\infty))\)} and let $K>0$.
{All pointwise inequalities below are understood for this
fixed representative.} Assume that there exists a full-measure set $\mathcal T\subset[0,\infty)$ with $0\in\mathcal T$ such that, for every $s\in\mathcal T$ and every $t\ge s$,
\[
y(t)+K\int_s^t y(\tau)\,d\tau\le y(s).
\]
Then
\[
y(t)\le y(s)e^{-K(t-s)}
\]
for every $s\in\mathcal T$ and every $t\ge s$.
\end{corollary}
This extends \cite[Proposition~A.1]{Kelliher2025} by removing the nonnegativity assumption on $y$.
\begin{proof}
Choose any $D'>0$ with $D'\neq K$ and any $C'>0$, and apply Proposition~\ref{AppendixA} with $q\equiv0$.
\end{proof}

\bibliographystyle{plain}
\bibliography{ref}

\end{document}